\documentclass[12pt]{amsart}
\usepackage[utf8]{inputenc}
\usepackage[english]{babel}
\usepackage[T1]{fontenc}
\usepackage{amsmath}
\usepackage{amsfonts}
\usepackage{amssymb}
\usepackage{amsthm}
\usepackage{amscd}
\usepackage{soul}
\usepackage{geometry}
\usepackage{enumitem} 
\usepackage{cancel} 
\usepackage{graphicx}
\usepackage{mathtools}
\usepackage{color}
\usepackage{comment}
\usepackage{color}
\definecolor{marin}{rgb}   {0.,   0.3,   0.7} 
\definecolor{rouge}{rgb}   {0.8,   0.,   0.} 
\definecolor{sepia}{rgb}   {0.8,   0.5,   0.} 
\usepackage[colorlinks,citecolor=marin,linkcolor=rouge,
            bookmarksopen,
            bookmarksnumbered
           ]{hyperref}

\newcommand\N{\mathbb{N}}
\newcommand\T{\mathbb{T}}
\newcommand\Z{\mathbb{Z}}

\newcommand\R{\mathbb{R}}
\newcommand\C{\mathbb{C}}

\newcommand{\Uc}{\mathcal{U}}
\newcommand{\Zc}{\mathcal{Z}}
\newcommand{\Hc}{\mathcal{H}}
\newcommand{\Bc}{\mathcal{B}}
\newcommand{\Pc}{\mathcal{P}}
\newcommand{\dd}{\mathrm{d}}

\newcommand{\enstq}[2]{\left\{#1~\middle|~#2\right\}}

\newcommand\eps{\varepsilon}

\newcommand{\Norm}[2]{\|#1\|\left.\vphantom{T_{j_0}^0}\!\!\right._{#2}}  

\newtheorem{theorem}{Theorem}

\newtheorem{lemma}{Lemma}

\newtheorem{remark}{Remark}

\title{Around plane waves solutions of the Schr\"odinger-Langevin equation}
\author{Quentin Chauleur}
\address{INRIA Rennes, Univ Rennes \& Institut de Recherche Math\'ematiques de Rennes,
CNRS UMR 6625 Rennes, Campus Beaulieu F-35042 Rennes Cedex, France. }
\email{Quentin.Chauleur@univ-rennes1.fr}

\author{Erwan Faou}
\address{INRIA Rennes, Univ Rennes \& Institut de Recherche Math\'ematiques de Rennes,
CNRS UMR 6625 Rennes, Campus Beaulieu F-35042 Rennes Cedex, France. }
\email{Erwan.Faou@inria.fr}

\begin{document}

\maketitle
\begin{abstract}
We consider the logarithmic Schr\"odinger equations with damping, also called Schr\"odinger-Langevin equation. On a periodic domain, this equation possesses plane wave solutions that are explicit. We prove that these solutions are asymptotically stable in Sobolev regularity. In the case without damping, we prove that for almost all value of the nonlinear parameter, these solutions are stable in high Sobolev regularity for arbitrary long times when the solution is close to a plane wave.  We also show and discuss numerical experiments illustrating our results. 
\end{abstract}


\section{Introduction}
Let $\T^d=\R^d/( 2 \pi \Z)^d $ denote the $d$-dimensional torus ($d \in \N^*$). We consider the logarithmic Schr\"odinger equation with damping, also called Schr\"odinger-Langevin equation,
\begin{equation}  i  \partial_t \psi +  \Delta \psi = \lambda \psi \log(|\psi|^2) + \frac{\mu}{2i}  \psi \log\left( \frac{ \psi}{\psi^*} \right),  \label{SL_eq} \end{equation} 
with $t\geq 0$, $x \in \T^d$, $\psi(0,x)=\psi^0(x)$, $\lambda <0$ (focusing case) or $\lambda > 0$ (defocusing case), and $\mu \geq 0$. We also denote $\psi^*$ (or also $\overline{\psi}$) the complex conjugate of a complex function $\psi$. Note that when $\mu = 0$, this equation is the 
logarithmic Schr\"odinger equation
\begin{equation} \label{logNLS}
 i \partial_t \psi + \Delta \psi = \lambda \psi \log |\psi|^2. 
\end{equation}
This latter equation was introduced in \cite{birula1976} as a model of nonlinear wave mechanics, and has then be proposed to model various phenomena such as quantum optics \cite{buljan2003}, \cite{krolikowski2000}, nuclear physics \cite{hefter1985} or transport and diffusion phenomena \cite{martino2007}, \cite{hansson2009}. On the other hand, the Schr\"odinger-Langevin equation \eqref{SL_eq} first appears in \cite{nassar1985} as a possible way to give a stochastic interpretation of quantum mechanics in the context of Bohmian mechanics. It had a recent renewed interest in the physics community, in particular in quantum mechanics in order to describe the continuous measurement of the position of a quantum particle (see for example \cite{nassar}, \cite{zander} or \cite{mousavi2019}) and in cosmology and statistical mechanics (see \cite{chavanis2017}, \cite{chavanis2019cosmo} or \cite{chavanis2019stat}). 

A lot of properties of these two equations are already known on the whole space $\R^d$. The mathematical study of equation \eqref{logNLS} for the focusing case $\lambda <0$ goes back to \cite{cazenave1983}, and the global existence of solutions is now well understood (see \cite{cazenave} and \cite{davenia2014}), as well as their qualitative behaviors (see for instance \cite{ardila2016}, \cite{ferriere2019} or \cite{ferriere2021multisoliton}). On the other hand, the defocusing case $\lambda > 0$ for \eqref{logNLS} has received some recent attention, in particular from the work \cite{carles2018} which established the global existence and the uniqueness of solutions as well as their asymptotic behavior. For the Schr\"odinger-Langevin equation \eqref{SL_eq}, the long-time behavior of solutions is given in \cite{chauleur2020} under some global existence assumptions.

However, few results are known about these equations on the $d$-dimensional torus geometry. The existence and uniqueness of global weak solutions to the logarithmic Schr\"odinger equation \eqref{logNLS} are given in \cite{carles_bao_DF}, and the existence of global dissipative solutions to the Euler-Langevin-Korteweg equations, which is the fluid counterpart of the Schr\"odinger-Langevin equation through the Madelung transform $\psi=\sqrt{\rho}e^{iS}$, is established in \cite{chauleur2021global}. As no behaviour properties or asymptotic features are currently known up to the authors knowledge, this paper is a first step in order to give some qualitative description of the solutions of these equations on $\T^d$. 

The most striking property of \eqref{SL_eq} is the preservation of the $L^2$ norm even in the damped case $\mu > 0$, hence this equation is always conservative in this sense. On $\T^d$, a first natural question is to analyze the existence and stability of plane waves solutions. They are of the form
\begin{equation} \label{general_plane_wave}
 \nu_{m}(t,x)=\rho e^{i m \cdot x - \omega t}   
\end{equation}
for $\rho >0$, $m \in \Z^d$ and $\omega \in \R$, they belong to any Sobolev space on the torus, and they usually play the role of ground states for nonlinear equations on $\T^d$. For instance, in the case of the classical cubic nonlinear Schr\"odinger
\begin{equation} \label{NLS}
	i \partial \psi + \Delta \psi = \lambda |\psi|^2 \psi,
\end{equation}
plane-wave solutions \eqref{general_plane_wave} correspond to $\omega= |m|^2 + \lambda \rho^2$, and 
it has been shown in  \cite{faou2012} that these solutions are stable in large Sobolev spaces $H^s$: to be more precise, small perturbations of \eqref{general_plane_wave} in $H^s$ remain essentially localized in the $m$-th Fourier mode over very long times in $H^s$ for sufficiently large Sobolev exponent $s$, and {\em almost all} $\rho$ which corresponds here to the $L^2$ norm of the plane wave\footnote{For the orbital stability in the energy space $H^1$, results can also be found in Zhidkov \cite[Sect.~3.3]{Zhidkov2001}  
and Gallay \& Haragus \cite{GH1,GH2}.}. 

In the case of the logarithmic Schr\"odinger equation \eqref{logNLS}, for initial data made of a single Fourier mode $u(0,x)=\rho e^{im \cdot  x}$ (with $\rho >0$ and $m \in \Z^d$), the equation has the unique plane-wave solution $\nu_m(t,x)= \rho e^{i(m \cdot x - \omega t)  }$ localized at the $m$-th Fourier mode, with 
\[\omega= |m|^2+2\lambda \log \rho. \]
As performed in \cite{faou2012}, a natural question is then to study the stability of these solutions in Sobolev spaces other very long time. Indeed, we state such a result in Theorem \ref{theorem_plane_wave}, and give a proof based on normal transformations as in \cite{grebert2007,bambusi2006}. The main difference with the result for \eqref{NLS} is that the frequencies of the linearized operator do not depend on $\rho$, by a subtle mechanism of scaling invariance of \eqref{logNLS}. This result is thus obtain for {\em almost all} $\lambda$, to avoid resonances in the equation. 

In contrast, the damping effect in the Schr\"odinger-Langevin induces that the only stationary plane wave solutions of \eqref{SL_eq} are the constant plane wave functions of the form
\begin{equation} \label{constant_SL}
 \nu= \rho e^{-2i \lambda \log \rho/\mu },   
 \end{equation}
where $\rho >0$.  We study the dynamics of perturbations of such a solution in Theorem~\ref{theorem_plane_wave_SL}. We will show that every small perturbations of (complex) constants \eqref{constant_SL} converges exponentially fast in $H^s(\T^d)$ with $s > \frac{d}2$ towards the constant solution \eqref{constant_SL} where $\rho$ denotes the $L^2$ norm of the solution, which is preserved by the dynamics. The rate of convergence depends on arithmetic relations between $\lambda$ and $\mu$ which can also generate Jordan block dynamics. Surprisingly, we also have situations where for given $\lambda$, the relaxation is more slow for large $\mu$ than for small $\mu$, and where the damping rate depend on the modes. This is exemplified by numerical experiments and proved in detail in Section 3.   

This paper is organized as follows. In Section 2, we will recall some notations and properties of functional analysis and nonlinear PDE analysis in order to state our results  Theorem \ref{theorem_plane_wave_SL} and Theorem \ref{theorem_plane_wave}. Then, Section 3 is dedicated to the proof of Theorem \ref{theorem_plane_wave_SL}, and Section 4 to the proof of Theorem \ref{theorem_plane_wave}. Finally, in Section 5, we show some numerical simulations in order to illustrate our results and explore situations not covered by our analysis. 

\section{Algebraic context and main results}
With a periodic function $u \in L^2(\T^d)$ we can associate the Fourier coefficients $u_n$ for $n = (n_1,\ldots,n_d) \in \Z^d$ defined by
\[ u_n= \frac{1}{(2\pi)^d}\int_{\T^d} u(x) e^{-in \cdot x} \dd x, \]
with $n\cdot x = n_1 x_1 + \cdots n_d x_d$, $x = (x_1,\ldots,x_d)$. For the average, we will use the specific notation 
$$
\langle u \rangle = u_0 = \frac{1}{(2\pi)^d} \int_{\T^d} u(x) \dd x. 
$$
We define the Sobolev space $H^s(\T^d)$ associated with the norm
\[ \Norm{u }{H^s}= \left(\sum_{n \in \Z^d }  \left(1+ |n|^2 \right)^{s} |u_n |^2 \right)^{\frac{1}{2}},  \]
where $|n|^2 = n_1^2 + \cdots + n_d^2$ for $n = (n_1,\ldots,n_d) \in \Z^d$. 
We recall that $H^s(\T^d)$ is an algebra when $s>d/2$, namely there exists a constant $C_s$ such that for all $u$, $v\in H^s$, we have 
\begin{equation}
\label{algebra}
\Norm{u v}{H^s} \leq C_s \Norm{u}{H^s} \Norm{v}{H^s}.
\end{equation}
Note also that we have 
$$
\Norm{u}{L^2}^2 = \sum_{n\in \Z^d}|u_n|^2 = \frac{1}{(2\pi)^d} \int_{\T^d} |u(x)|^2 \dd x = (u,u)_{L^2},
$$
where 
$$
(u,v)_{L^2} = \frac{1}{(2 \pi)^d} \int_{\T^d} \overline{u(x)}{v(x)} \dd x. 
$$
We now define the notion of solution to the Schr\"odinger-Langevin equation. For $T > 0$, 
and an application $t \mapsto u(t,x) \in \mathcal{C}([0,T], H^s)$ such that $\langle u \rangle \neq 0$, by classical lifting theorem, we can define $a(t) > 0$ and $\theta(t) \in \R$ such that $\langle u(t) \rangle = a(t) e^{i \theta(t)}$, and such that the application $t \mapsto (a(t),\theta(t))$ is continuous on $[0,T]$.
We define the function 
$$
w =  a - e^{-i \theta} u = |\langle u \rangle | \left( 1 - \frac{u}{\langle u \rangle}\right),
$$
so that we have the following parametrization, valid for all $u$ such that $\langle u \rangle \neq 0$: 
\begin{equation}
\label{proj}
u(t,x) = e^{i \theta(t)} (a(t) + w(t,x)), 
\end{equation}
where $\theta(t)$, $a(t)$ and $w(t,x)$ are continuous in time, as long as $\langle u \rangle \neq 0$. 
In this case, we can define the logarithm
$$
\log(u (t,x) ) :=  i \theta(t) + \log a(t) + \log \left(1 + \frac{w(t,x)}{a(t)}\right). 
$$
This application is well defined and smooth for curves on the domain 
\begin{multline*}
\mathcal{U}_s =  \enstq{u = e^{i \theta}(a + w)}{(a,\theta,w) \in \R_+ \times \T \times H^s,\,  a > 0, \, \langle w \rangle = 0, \quad \Big\|\frac{w}{a}\Big\|_{H^s} <  \frac{1}{C_s} }
\end{multline*}
where $C_s$ is the constant appearing in \eqref{algebra}. Note that this set contains arbitrary large functions as both $a$ and $w$ can become large. 

On $\mathcal{U}_s$, and owing to the analytic series 
\begin{equation}
\label{lelog}
\log \left(1 + \frac{w}{a}\right) =  - \sum_{n \geq 1} \frac{1}{n}\left(- \frac{w}{a }\right)^{n},
\end{equation}
we see that the application $u \mapsto \log u$ defined above is analytic. 
With this definition of the logarithm, it is clear that for a curve $t \mapsto u(t) \in \Uc_s$ we have $u^*(t) \in \Uc_s$, and $\log |u(t)|^2 = \log u(t) + \log u^*(t)$. Moreover, we have $\log(u^*(t)) = (\log u(t))^*$. 
%
%
%
Hence we deduce that  $u(t)  \mapsto \lambda u (t)\log |u(t)|^2 = \lambda u(t) (\log u(t) + \log u^*(t))$ is well defined on $\mathcal{C}([0,T],\Uc_s)$. Similarly, the function $u(t) \mapsto \log(u(t)) - \log(u^*(t))$ is well defined on $\mathcal{C}([0,T],\Uc_s)$, and hence so is the function \[u\in\mathcal{C}([0,T],\Uc_s)\mapsto \frac{\mu}{2i}  u \log\left( \frac{ u}{u^*} \right) \in \mathcal{C}([0,T],H^s).\]
This thus allows to define mild and strong solutions to the Schr\"odinger-Langevin equation \eqref{SL_eq} than can be written 
\begin{equation}  
i  \partial_t \psi +  \Delta \psi = z \psi \log(\psi) + z^* \psi \log\left( \psi^*  \right), \quad \mbox{with} \quad z = \lambda + \frac{\mu}{2i}.\label{SL_eq2} \end{equation} 

Note that the $L^2$ norm is preserved along the dynamics of this equation, as we can check that 
$$
\frac{\dd}{\dd t} \Norm{\psi}{L^2}^2 = 2 \mathrm{Re} ( \psi, \partial_t \psi)_{L^2}
= - 2 \, \mathrm{Im}  ( \psi, \Delta \psi)_{L^2}  + 4\, \mathrm{Im}  (\psi, \psi  \, \mathrm{Re}(z \log \psi) )_{L^2} = 0. 
$$

We now state our first theorem concerning the Schr\"odinger-Langevin equation: 

\begin{theorem} \label{theorem_plane_wave_SL}
Let $s > \frac{d}{2}$,  $\lambda > -\frac{1}{2}$, $\mu > 0$ and $\rho > 0$. Then 
there exists $\eps_0>0$ such that, if the initial datum $\psi^0 \in \mathcal{U}_s$ satisfies
$$
\Norm{\psi^0 - \langle \psi^0\rangle}{H^s} \leq \varepsilon_0 \quad \mbox{and}\quad \Norm{\psi_0}{L^2}  = \rho
$$
then the solution of \eqref{SL_eq} with $\psi(0,.)=\psi^0 \in \Uc_s$ satisfies for $t \geq 0$, 
\begin{equation}
\label{rateSL} \Norm{  \psi(t,.) - \rho e^{-2i \lambda \log \rho/\mu }}{H^s} \leq C e^{- \alpha t}(1 + \beta t), 
\end{equation}
where the constant $C>0$ depends on $s$, $d$, $\rho$ and $\varepsilon_0$, and where $\alpha >0 $  and $\beta \in \{0,1\}$ depend on $\mu$ and $\lambda$ as follows: 
\begin{itemize}
\item[(i)]  If $\mu < 2 \sqrt{1+2 \lambda}$ then $\alpha = \frac{\mu}{2}$ and $\beta = 0$. 
\item[(ii)] If $\mu = 2 \sqrt{1 + 2 \lambda}$ then $\alpha = \frac{\mu}{2}$ and $\beta = 1$. \item[(iii)] If $\mu > 2 \sqrt{1+2 \lambda}$ then 
$$
\alpha = \frac{\mu}{2} - \Big(\frac{\mu^2}{4}- 1- 2 \lambda  \Big)^{\frac12} \in \left(0,\frac{\mu}{2}\right). 
$$
and if there exists $n \geq 2$ such that $\mu^2 = 4n^2 + 8 \lambda n$ we have  $\beta =1$, and if this is not the case, $\beta = 0$. 

\end{itemize}
Moreover, for any fixed $j \in \Z^d\backslash\{0\}$, we have 
\begin{equation}
\label{psij}
|\psi_j(t) | \leq \frac{C}{j^s} e^{- \alpha_j t}( 1 + \beta_j t), 
\end{equation}
where 
$$
\alpha_j =   \frac{\mu}{2} -    \max\big(0, \frac{\mu^2}{4} - |j|^4- 2 \lambda |j|^2 \big)^{\frac12}
$$
and $\beta_j = 0$ unless $\frac{\mu^2}{4} - |j|^4- 2 \lambda |j|^2 = 0$ and in this case $\beta_j = 1$. 
\end{theorem}
Let us make the following remarks: 
\begin{itemize}
\item The presence of $\beta = 1$ reflects the possibility of Jordan block in the reduced dynamics. 

\item There is no restriction on $\rho$, the plane wave solution can be arbitrarily large. 

\item When $\lambda$ is fixed, $\mu \to \infty$, the damping rate 
$$
\alpha = \frac{2\lambda + 1}{\mu} + \mathcal{O} \left( \frac{1}{\mu^3} \right)
$$
 goes to $0$. Thus a larger damping coefficient implies a slowlier relaxation to the equilibrium. This echoes some known behaviors in other context, see for instance \cite{herda2018} and the reference therein. To our knowledge, such behaviour was however never identified in conservative models coming from quantum mechanics. 

\item Note that the damping rate $\alpha_j$ decreases with the size of the mode. Hence the smallest damping rate is always given by $\alpha_j$ with $|j| = 1$, but the modes decay at speed depending on $|j|$.
This multiscale decomposition is confirmed by numerical experiments (see in particular Figure \ref{fig:mu_8_actions.png}). 
\end{itemize}
\medskip 

We now state our second result, which analyzes the case $\mu = 0$, that it the logarithmic Schr\"odinger equation \eqref{logNLS}. Recall that 
for $m \in \Z^d$, plane wave solutions of \eqref{logNLS} are explicitly given by 
\begin{equation} \label{plane_wave_logNLS}
 \nu_m(t,x)=\rho e^{i \theta_0} e^{im \cdot x} e^{-i\left(|m|^2+2 \lambda \log \rho  \right) t} 
 \end{equation}
for all $t \geq 0$ and $x \in \T^d$, and $\rho=\Norm{ \nu_m(t,x) }{L^2(\T^d)}>0$. We prove the stability of such functions by $H^s$ perturbations. 
In this case we use a {\em normal form} strategy already employed in the cubic case (see \cite{faou2012}) which ensures the $H^s$-stability of plane waves solutions of \eqref{logNLS} for sufficiently large Sobolev exponent $s$, and in the following sense: 
\begin{theorem} \label{theorem_plane_wave} Let $m \in \Z^d$ be fixed. 
Let $-\frac{1}{2}<\lambda_-<\lambda_+$, $\rho > 0$, $\theta_0 \in \R$ and $N>1$ be fixed arbitrarily. 
Then there exist $s_0>0$, $C \geq 1$ and a set of full measure $\Lambda$ in the interval $\left[ \lambda_-,\lambda_+ \right]$ such that for every $s \geq s_0$ and every $\lambda \in \Lambda$, there exists $\eps_0>0$ such that the following holds: if the initial data $\psi^0 \in \Uc_s$ is such that 
$$
\Norm{ e^{-im.x} \psi^0 - \langle e^{-im.x} \psi^0\rangle}{H^s}=\eps \leq \eps_0,  \quad
\mbox{and} \quad 
\Norm{\psi^0}{L^2} = \rho, 
$$
then the solution $\psi(t)$ of \eqref{logNLS} with $\psi(0,x) =\psi^0(x) $ satisfies
\[ \Norm{ e^{-im.x} \psi(t,\,\cdot\,) - \rho e^{i \theta(t)}}{H^s} \leq C \eps \ \ \ \text{for} \ \ \ t\leq \eps^{-N}, \]
where $\theta(t)$ is such that 
$$
|\dot \theta + |m|^2 - 2 \lambda \log \rho | \leq C \varepsilon^2. 
$$
\end{theorem}
The proof relies on a Birkhoff normal form Theorem (see \cite{grebert2007} and \cite{bambusi2006}), which will be recalled in Section 3.6. with Theorem \ref{theorem_normal_form}.\\

\section{Asymptotic stability for Schr\"odinger-Langevin equation}
\subsection{Elimination of the zero mode}
Note that  if $\psi$ is solution of \eqref{SL_eq} with initial datum $\psi(0,x)=\psi^0$, then the purely time-dependent gauge transform
\begin{equation}
\label{gauge1} \kappa \psi \exp \left(2i \frac{\lambda}{\mu} \log \kappa \left(1- e^{-\mu t} \right) \right) \end{equation}
is also a solution of \eqref{SL_eq} with initial datum $\psi(0,x)=\kappa \psi^0$. Hence by taking $\kappa = \rho^{-1}$ and modifying $\varepsilon_0$ and $C$ accordingly, we see that it is sufficient to prove the result for $\rho = 1$. 

We search a solution under the form 
\[ \psi(t,x)=e^{i \theta(t)} (a(t) +w(t))\]
with $a(t)>0$, $\theta(t) \in \R$, and 
\[ w(t,x)= \sum_{j \neq 0} w_j(t) e^{i j \cdot x} \]
on the Fourier basis. For $r_0 > 0$, we denote 
$$
\Bc_s(r_0) = \{ w \in H^s\, | \, \langle w \rangle = 0, \quad \mbox{and} \quad \Norm{w}{H^s} \leq r_0\}. 
$$
Note that we have $\psi - \langle \psi \rangle = e^{i\theta} w$ and hence we can assume $w(0,.) \in \Bc_s(r_0)$ at time $t=0$, for $r_0$ small enough. 
We also recall that the $L^2$ norm of the solution is $\Norm{u}{L^2}^2 = a^2 + \Norm{w}{L^2}^2 = \rho^2$, and it is preserved along the dynamics, hence we have
\begin{equation} \label{parseval}
 a = \sqrt{ 1 - \Norm{w}{L^2}^2}.   
 \end{equation}
Now we calculate with $\psi = e^{ i \theta} (a + w)$, using the representation \eqref{SL_eq2} with $z = \lambda + \frac{\mu}{2i}$: 
\begin{align*}
&i \partial_t \psi = -e^{i \theta } (a + w) \dot \theta  + e^{i \theta }i \dot a + e^{i\theta}i \partial_{t} w, \\
&\Delta \psi = e^{i \theta}\Delta w, \\
& z  \psi \log(\psi)  = z e^{i \theta}( a + w) \left(i \theta + \log a + \log \left(1 + \frac{w}{a}\right) \right),\\
& z^*  \psi \log(\psi^*) = z^{*} e^{i \theta}( a + w) \left(-i \theta + \log a + \log \left(1 + \frac{w^*}{a}\right) \right). 
\end{align*}
We then see that \eqref{SL_eq2} is equivalent to 
\begin{multline*}
- (a + w) \dot  \theta  + i \dot a + i \partial_{t} w  + \Delta w = \\
z ( a + w) \left(i \theta + \log a + \log \left(1 + \frac{w}{a}\right) \right) + z^{*} ( a + w) \left(-i \theta + \log a + \log \left(1 + \frac{w^*}{a}\right)\right),
\end{multline*}
or equivalently 
\begin{equation}
\label{pifpaf}
- (a + w)(  \dot \theta  + \mu \theta + 2 \lambda \log(a) )    + i \dot a + i \partial_{t} w  + \Delta w = z w + z^* w^* + X(w,w^*),
\end{equation}
where 
\begin{align*}
X(w,w^*) &:= z ( a + w)  \log \left(1 + \frac{w}{a}\right) - z w  + z^{*} ( a + w) \log \left(1 + \frac{w^*}{a}\right)- z^* w^* \\
&= - z ( a + w)  \sum_{n \geq 1} \frac{1}{n}\left(- \frac{w}{a }\right)^{n}- z w - z^* ( a + w) \sum_{n \geq 1} \frac{1}{n}\left(- \frac{w^*}{a }\right)^{n} - z^*w^*.
\end{align*}
Hence, as $a$ is given by \eqref{parseval}, we have, for $r_0$ small enough,
$$
w \in \Bc_s(r_0) \quad \Longrightarrow \quad \Norm{X(w,w^*)}{H^s} \leq C_{r_0} \Norm{w}{H^s}^2. 
$$
Taking the average $\langle \, \cdot \, \rangle$ of \eqref{pifpaf}, we obtain	
$$
- a (  \dot \theta  + \mu \theta + 2 \lambda \log(a) )    + i \dot a  = \langle X(w,w^*) \rangle, 
$$
which yields an equation for $\theta$:
\begin{equation}
\label{thetamu}
\dot \theta  + \mu \theta + 2 \lambda \log(a) = \frac{1}{a} \mathrm{Re}\langle X(w,w^*) \rangle.
\end{equation}
We also note that as long as $w \in \Bc_s(r_0)$, we have 
$$
\left|\frac{1}{a} \mathrm{Re}\langle X(w,w^*) \rangle\right|\leq C \Norm{w}{H^s}^2. 
$$
Projecting again \eqref{pifpaf} on the non zero modes, we obtain 
$$
i \partial_t w + \Delta w = z w + z^* w^* + X(w,w^*) - \langle X(w,w^*)\rangle + w (  \dot \theta  + \mu \theta + 2 \lambda \log(a) ).$$
This shows that the equation for $w$ can be written 
\begin{equation}
\label{eqw}
i \partial_t w = -\Delta w + z w + z^* w^* + F(w,w^*),
\end{equation}
where 
$$
F(w,w^*) = X(w,w^*) - \langle X(w,w^*)\rangle + \frac{w}{a}   \mathrm{Re}\langle X(w,w^*)\rangle
$$
satisfies the estimate 
$$
w \in \Bc_s(r_0) \quad \Longrightarrow \quad \Norm{F(w,w^*)}{H^s} \leq C \Norm{w}{H^s}^2. 
$$

\subsection{Diagonalization and Jordan blocks}
To analyze the long time behavior of \eqref{eqw}, we embed it into the complex system 
\begin{equation}
\label{eqw2}
\left|
\begin{array}{rcl}
i \partial_t \xi &=& -\Delta \xi + z \xi + z^* \eta + F(\xi,\eta)\\ 
i \partial_t \eta &=& \Delta \eta - z^* \eta - z \xi - F(\xi,\eta)^*
\end{array}
\right. 
\end{equation}
and we consider this system in the ball $(\xi,\eta) \in \Bc_s(r_0) \times \Bc_s(r_0)$ equipped with the norm product 
$$
\Norm{(\xi,\eta)}{H^s}^2 = \Norm{\xi}{H^s}^2 + \Norm{\eta}{H^s}^2.
$$ 
We can write this system as 
\begin{equation}
\label{eqw4}
i \partial_t \begin{pmatrix}
\xi \\ \eta
\end{pmatrix} = \begin{pmatrix} -\Delta + z & z^* \\ -z & \Delta - z^* \end{pmatrix}\begin{pmatrix}
\xi \\ \eta
\end{pmatrix} + Q(\xi,\eta), 
\end{equation}
where $Q: \Bc_s(r_0) \times \Bc_s(r_0) \to H^s \times H^s$ satisfy the estimate 
$$
\Norm{Q(\xi,\eta)}{H^s} \leq C \Norm{(\xi,\eta)}{s}^2.
$$
Now, if we consider the extended system \eqref{eqw2} with an initial condition satisfying 
$$
\xi(0) = \eta(0)^*
$$
then for all times $\xi(t) = \eta(t)^*$ and  $\xi(t) = w(t)$  coincides with the solution of \eqref{eqw}. Note that in Fourier, this condition yields $\xi_j = (\eta^*)_j = (\eta_{-j})^*$. 

Taking the Fourier transform of the equation \eqref{eqw4}, we obtain the collection of equations
\begin{equation}
\label{eqw44}
\forall\, j \in \Z^d\backslash \left\{ 0 \right\}, \quad 
i \partial_t \begin{pmatrix}
\xi_j \\ \eta_j
\end{pmatrix} = A_{|j|^2}\begin{pmatrix}
\xi_j \\ \eta_j
\end{pmatrix} + Q_j(\xi,\eta). 
\end{equation}
where for $n = |j|^2$, 
\[ A_n= \left( \begin{array}{cc} n + z  & z^*  \\
                                - z & -n- z^*                      
                        \end{array} \right) =  \left( \begin{array}{cc} n + \lambda +\frac{\mu}{2i} & \lambda -\frac{\mu}{2i} \\
                                - \lambda - \frac{\mu}{2i} & -n-\lambda +\frac{\mu}{2i}                       
                        \end{array} \right). \] 
In particular, the eigenvalues of the matrix $A_n$ are 
$$
\frac{\mu}{2i } \pm \sqrt{n^2+2 \lambda n - \frac{\mu^2}{4}}.
$$
We can distinguish three cases:

\noindent {\em (i)} \ul{$4 n^2+8 \lambda n - \mu^2  > 0$}. In this case the eigenvalues are of the form $\frac{\mu}{2i} \pm \delta_n$, with 
$$
\delta_n  = \sqrt{n^2+2 \lambda n - \frac{\mu^2}{4}} > 0. 
$$
We can explicit the diagonalization $A_n = P_n D_n P_n^{-1}$, with the diagonal matrix
\[ D_n=\left( \begin{array}{cc} 
\frac{\mu}{2i} - \delta_n & 0 \\
0 & \frac{\mu}{2i} + \delta_n \end{array} \right), \]
and the matrix of change of coordinates
\begin{equation}
\label{eqP-1} P_n^{-1}=\frac{1}{\sqrt{2 \delta_n (\delta_n + n + \lambda)}}\left( \begin{array}{cc}  n + \lambda + \delta_n &
        \lambda - \frac{\mu}{2i} \\
         \lambda + \frac{\mu}{2i} &
         n + \lambda + \delta_n 
        \end{array} \right) 
\end{equation}
and
\begin{equation}
\label{eqP} P_n=\frac{1}{\sqrt{2 \delta_n (\delta_n + n + \lambda)}}\left( \begin{array}{cc}
         n + \lambda + \delta_n &
        - \left( \lambda - \frac{\mu}{2i} \right) \\
         - \left(\lambda + \frac{\mu}{2i} \right) &
         n + \lambda + \delta_n 
        \end{array} \right).
\end{equation}
Note that $P_n$ is hermitian, has condition number smaller than 2 and that
\[ \det \left( P_n \right) =  \det \left( P_n^{-1} \right) =1.  \]

\noindent {\em (ii)} \ul{$4 n^2+8 \lambda n - \mu^2  < 0$}. Note that for given $\lambda$ and $\mu$, this situation occurs only a finite number of times, as when $n$ becomes large, $n^2 + 2\lambda n$ goes  to $+\infty$. In this case the two eigenvalues are under for the form $-i\alpha_n$ and $-i \beta_n$, with 
\begin{align*}
&\alpha_n = \frac{\mu}{2} - \sqrt{ \frac{\mu^2}{4} - n^2- 2 \lambda n}, 
&\beta_n = \frac{\mu}{2} + \sqrt{ \frac{\mu^2}{4} - n^2- 2 \lambda n}.
\end{align*}
Now as $n \geq 1$ and $\lambda > -\frac12$, we have that $n^2+ 2 \lambda n > 0$ and hence $\beta_n >\alpha_n >0$. Let us finally note that the sequence $n^2 + 2 \lambda n$ is increasing, so the minimum of all the $\alpha_n$ and $\beta_n$ is $\alpha_1$, and this case happen if $4 +8 \lambda  - \mu^2  < 0$. 
The diagonalization matrices are the same as in \eqref{eqP-1} and \eqref{eqP}, and we have 
$$
P_n^{-1} A_n P_n = \begin{pmatrix} - i \beta_n  & 0 \\ 0 & - i \alpha_n \end{pmatrix},
$$
but note that as $\delta_n = i (\frac{\mu^2}{4} - n^2 - 2 \lambda n)^{\frac12}$ is now imaginary, these matrices are not hermitian anymore and have no special geometric structure.

\noindent {\em (iii)} \ul{$4 n^2+8 \lambda n - \mu^2  = 0$}. In this case $\frac{\mu}{2i}$ is a double eigenvalue. As 
$$(n + \lambda)^2 = ( \lambda - \frac{\mu}{2i}) ( \lambda + \frac{\mu}{2i}),$$ 
we have that $( n+ \lambda, - \lambda - \frac{\mu}{2i})$ is an eigenvector 
and the matrix $A_n$ can be put under the Jordan form  
$$
P_{n}^{-1} A_n P_n = \begin{pmatrix}  \frac{\mu}{2i} &  \lambda - \frac{\mu}{2i} \\ 0 &  \frac{\mu}{2i} \end{pmatrix},
$$
with 
$$
P_n^{-1} = \frac{1}{n + \lambda}\begin{pmatrix}
2(n + \lambda) & \lambda - \frac{\mu}{2i} \\ \lambda + \frac{\mu}{2 i} &   (n + \lambda) 
\end{pmatrix}
$$
and
$$
P_n = \frac{1}{n + \lambda}\begin{pmatrix}
n + \lambda & -\lambda + \frac{\mu}{2i} \\ - \lambda - \frac{\mu}{2 i} &  2 (n + \lambda) 
\end{pmatrix}.
$$ 

Again, the matrices $P_n$ and $P_n^{-1}$ have condition number smaller than $2$, uniformly in $n$. 

\medskip
We now go back to \eqref{eqw4}, and we define the operator $V = \begin{pmatrix} f \\g \end{pmatrix}= P^{-1} \begin{pmatrix} \xi \\ \eta\end{pmatrix}$ in Fourier by the formula 
$$
\forall\, j \in \Z^d \backslash \{0\}, \quad 
V_j = \begin{pmatrix} f_j \\g_j \end{pmatrix} = P_{|j|^2}^{-1} \begin{pmatrix} \xi_j \\ \eta_j\end{pmatrix}.
$$
From the properties of the matrices $P_n$ exhibited in the previous three cases, the application $(\xi,\eta) \mapsto V$ is bounded and invertible in $H^s \times H^s$, the inverse being given by the multiplication with the matrix $P_{|j|^2}$.  Hence the system becomes 
$$
 \quad  i \partial_t V_j   = D_{|j|^2} V_j + R_j(V),
$$
where $R(V) = P^{-1} Q (P V)$ satisfies 
$$
\Norm{R(V)}{H^s} \leq C \Norm{V}{H^s}^2
$$
for $V \in \Bc_s (r_0) \times \Bc_s(r_0)$ and for some $r_0$ small enough. Here, the two by two matrices $D_{|j|^2}$ are given explicitly in term of $\lambda$, $\mu$ and $|j|^2$. 

Let $\alpha > 0$. 
We now define the operator $\Lambda$ as 
$$
(\Lambda V)_j = \alpha V_j \quad \mbox{if} \quad 4 n^2+8 \lambda n - \mu^2  \neq 0
$$
and 
$$
(\Lambda V)_j =  \begin{pmatrix}  \alpha &  -i \lambda + \frac{\mu}{2} \\ 0 &  \alpha \end{pmatrix} V_j \quad \mbox{if} \quad 4 n^2+8 \lambda n - \mu^2  = 0.
$$
We calculate that in this latter case, we have for $t \in \R$, 
$$
e^{t \Lambda_j }\begin{pmatrix} f_j \\g_j \end{pmatrix} =  \begin{pmatrix}  1  &  t(-i \lambda + \frac{\mu}{2}) \\ 0 &  1 \end{pmatrix} \begin{pmatrix} f_j \\g_j \end{pmatrix}.
$$
We can define the change of variable $U = e^{t \Lambda} V$ so we have 
\begin{align*}
i \partial_t U &= e^{t \Lambda} (i \Lambda  + D) e^{-t\Lambda} U + e^{t \Lambda} R(e^{- \Lambda t} U)\\  
&= (i \Lambda  + D) U + e^{t \Lambda} R(e^{- \Lambda t} U), 
\end{align*}
where the last equality comes from the fact that if $4 n^2+8 \lambda n - \mu^2  \neq 0$ the block operator $e^{t \Lambda_j}$ is the scalar multiplication by $e^{t\alpha}$, and in the Jordan block case, 
\begin{equation}
\label{Michael}
i \Lambda_j + D_{|j|^2} = \begin{pmatrix}
i (\alpha - \frac{\mu}{2}) & 0 \\ 0 & i (\alpha - \frac{\mu}{2})  
\end{pmatrix}.
\end{equation}
This leads to the following estimate: when $U \in \Bc_s(r_0)$, as $R$ is at least quadratic, we have that for $ t \geq 0$, 
$$
\Norm{e^{t \Lambda} R(e^{- \Lambda t} U)}{H^s} \leq C e^{- \alpha t} (1 + \beta t^3) \Norm{U}{H^s}^2
$$
where $\beta = 0$ in the case where $4 n^2+8 \lambda n - \mu^2  \neq 0$ for all $n$, {\em i.e.} when no Jordan block is present. 

Let us examine the operator $i \Lambda + D$. In the Jordan block case {\em (iii)}, we have seen that the operator is given by \eqref{Michael}. In the case $(i)$, it is 
$$
i \Lambda_j + D_{|j|^2} = \begin{pmatrix}
i ( \alpha - \frac{\mu}{2})  - \delta_n  & 0 \\ 0 & i ( \alpha - \frac{\mu}{2})  + \delta_n . 
\end{pmatrix}
$$
and in the case {\em (ii)},
$$
i \Lambda_j + D_{|j|^2} = \begin{pmatrix}
i ( \alpha - \beta_n)   & 0 \\ 0 & i ( \alpha - \alpha_n)   
\end{pmatrix}.
$$
Hence, if $\alpha =  \min( \alpha_n, \frac{\mu}{2}) = \min( \alpha_1, \frac{\mu}{2}) $, we have that
$$
i \Lambda + D =  - i A + B,
$$
where $A$ and $B$ are diagonal with real coefficients, the coefficients of $A$ being nonnegative. 

Let $L$ be the operator defined on $H^s \times H^s$ by the formula 
$$
(L U)_j = |j| U_j. 
$$
As $i \Lambda + D$ is diagonal, it commutes with $L$. 
Then we have for $U = \begin{pmatrix} f \\ g \end{pmatrix}$: 
\begin{align*}
\Norm{U}{H^s}^2 &= \sum_{j \neq 0} |f_j|^2 |n|^{2s} +  \sum_{j \neq 0} |g_j|^2 |n|^{2s}\\
& = ( L^s U, L^{s} U)_{L^2} = (  U, L^{2s} U)_{L^2}
\end{align*}
where for $U = \begin{pmatrix} f \\ g \end{pmatrix}$ and $V = \begin{pmatrix} \xi \\ \eta \end{pmatrix}$, 
$$
( U, V)_{L^2} = \frac{1}{2\pi} \int_{\T^d} \overline{f(x)} \xi(x) \dd x + 
\frac{1}{2\pi} \int_{\T^d} \overline{g(x)} \eta(x) \dd x .
$$
Let us calculate 
\begin{align*}
\frac{\dd}{\dd t}\Norm{U}{H^s}^2  &= 2\mathrm{Re}  ( L^s U, L^{s} \partial_t  U)_{L^2} 
\\
&=  2\mathrm{Re}   ( L^s U, L^{s} ( -  A -i  B) U)_{L^2} + K(t,U)\\
&= - 2\mathrm{Re}   ( L^s U, A L^{s}  U)_{L^2} + K(t,U), 
\end{align*}
where we have the estimate 
$$
|K(t,U)| \leq C e^{- \alpha t} (1 + \beta t^3) \Norm{U}{H^s}^3 .
$$

Moreover, with the choice of $\alpha$, we have 
$$
- 2\mathrm{Re}   ( L^s U, A L^{s}  U)_{L^2} \leq 0, 
$$
so we end up with the estimate 
$$
\left|\frac{\dd}{\dd t}\Norm{U}{H^s}^2 \right| \leq C e^{- \alpha t} (1 + \beta t^3) \Norm{U}{H^s}^3\leq M e^{- \frac{\alpha}{2} t}\Norm{U}{H^s}^3
$$
for some positive constant $M$, which is valid as long as $U(t) \in \Bc_s(r_0) \times \Bc_s(r_0)$ with $r_0$ small enough. 
If we denote $y = \|  U \|^2_{H^s(\T^d)}$, we have to study the differential inequation
\[  | \dot{y} | \leq M  e^{-\frac{\alpha}{2}t} y^{3/2}.    \]
We look at the following differential equation, for all $t \geq0$,
\[  \dot{f}(t) = M e^{-\frac{\alpha}{2}t} f(t)^{3/2}.  \]
Note that as the function $(t,f(t)) \mapsto e^{-\frac{\alpha}{2}t}f(t)^{3/2}$ is locally lipschitz with respect to its second variable $f$, the classical theory of Cauchy-Lipschitz theorem applies, and as $f(0)=y(0)=\Norm{ U(0)}{H^s}^2=\Norm{ V(0)}{H^s}^2>0$, we know that
$f(t)>0$ as long as the solution $f$ exists. In particular, we can write

\[ \frac{\dot{f}(t)}{ f(t)^{3/2}}= \frac{\dd}{\dd t} \left( \frac{-2}{ \sqrt{f(t)}} \right)=C_s  e^{-\frac{\alpha}{2}t},  \]
so
\[  \sqrt{f(t)} = \frac{1}{ \frac{1}{\sqrt{f(0)}}+ \frac{M}{\alpha} \left( e^{-\frac{\alpha}{2}t} -1 \right) }  \]
by integrating, so $f$ remains uniformly bounded under the condition
\[ f(0) < \left(\frac{\alpha}{M}\right)^2,    \]
and more precisely
\[  f(0) \leq f(t) \leq \frac{f(0)}{\left( 1 - \frac{M}{\alpha}\sqrt{f(0)}  \right)^2} < \infty,   \]
in particular the solution $f$ do not blow-up in finite time and exists for all time $t \geq 0$. As $f(0)=\Norm{  U(0)}{H^s}^2$, we have by comparison that $y(t) \leq f(t)$ for all $t\geq0$, so 
\[  \sup_{t \geq 0}  \Norm{  U(t)}{H^s}^2 < \infty \]
under the condition
\[ \Norm{  U(0)}{H^s} < \frac{\alpha}{M}.   \]
Recalling that $U(t)=e^{\Lambda t} V(t)$, we get that
\[ \Norm{ V(t)}{H^s} \leq C \big(\mu,s,d,\Norm{V(0)}{H^s} \big) e^{-\alpha t} ( 1 + \beta t),   \]
and this shows Theorem \ref{theorem_plane_wave_SL}, $\beta = 1$ or $\beta = 0$ reflecting the presence of Jordan block in the dynamics. Estimate \eqref{psij} is easily obtained by looking at the $j$-th block. 

Now if the initial conditions $\xi^* = \eta$ are satisfied, it gives a solution of the initial system. It remains to control $a$ and $\theta$, which is read in a non trivial relation between the components of $U$ and $V$, but without affecting the fact that $V \in \Bc_s$. But from \eqref{parseval} we obtain directly 
$$
| a - 1| \leq C \Norm{V}{H^s}^2 \leq C e^{-2 \alpha t} ( 1 + \beta t)^2, 
$$
and from \eqref{thetamu}
$$
|\dot \theta + \mu \theta| \leq  C e^{-2 \alpha t} ( 1 + \beta t)^2 \leq C e^{- \frac32 \alpha t}.
$$
We deduce that 
$$
|\theta(t)| \leq C e^{- \mu t }\int_{0}^t e^{(  \mu - \alpha) s } \dd s = \frac{1}{\mu - \frac32 \alpha}( e^{ - \frac32 \alpha t } - e^{- \mu t}) 
$$
and this shows the result, as $\alpha \leq \frac{\mu}{2}$.

\section{Stability for the logarithmic Schr\"odinger equation}

We now prove Theorem \ref{theorem_plane_wave}. The strategy follows the lines of \cite{faou2012} and uses a Birkhoff normal form reduction as in \cite{bambusi2003,bambusi2006}. As in the previous section, we can eliminate the zero mode and perform block diagonalization of the linear operator, but the preservation of the Hamiltonian structure is crucial. Indeed, the logarithmic Schr\"odinger equation is associated with the energy
\begin{equation}
\label{hamil}
    H(\psi,\psi^*) :=  \Norm{\nabla \psi}{L^2}^2 + \frac{\lambda }{(2\pi)^d}\int_{\T^d} |\psi(t,x)|^2 \left( \log |\psi(t,x)|^2 - 1\right)  dx,  
\end{equation}   which is preserved for all time $t \geq 0$, as equation \eqref{logNLS} can be written
$$
i \partial_t \psi = \frac{\partial H}{\partial \psi^*}( \psi,\psi^*)
\quad\mbox{and hence} \quad 
 \frac{\dd}{\dd t}H(\psi(t),\psi(t)^*)=0. 
 $$

\subsection{Reduction to the case $m=0$}
Equation \eqref{logNLS} satisfies the galilean invariance principle, which means that if $\psi$ is a solution, then
\[ \varphi (t,x)=\psi(t,x-vt) e^{-i \left( |v|^2 t/2 - v \cdot x \right)}    \]
is also a solution for every $v \in \Z^d$. Using this property on the plane wave $\nu_m$ with $v=-m$, we get that
\[ \nu_m(t,x +m t) e^{-i \left( |m|^2 t +m \cdot x \right)}  = \rho e^{-2i \lambda t \log \rho} = \nu_0(t,x),   \]
which means that we can restrict our attention to the case $m=0$. Note that this transformation preserves the $L^2$-norm. \\

Another important feature of \eqref{logNLS} is the effect of a scaling factor (cf \cite{carles2018}) : unlike what happens in the case of the usual power-like nonlinearity $\lambda \psi |\psi|^{\gamma}$ for $\gamma>1$, if $\psi$ is a solution to \eqref{logNLS}, then
\[  \kappa \psi(t,x) e^{2it\lambda \log \kappa} \]
also solves \eqref{logNLS} with initial datum $\kappa \psi^0$ (compare \eqref{gauge1}). Keeping this property in mind, and using the conservation of the $L^2$ norm of any solution of \eqref{logNLS}, we will take in the following, for all $t \geq 0$,  
\[ \Norm{ \psi(t)}{L^2} = \Norm{ \psi^0}{L^2}= 1, \]
which means that we can consider only the case $\rho = 1$. 
\subsection{Elimination of the zero mode}
We perform the same decomposition as in the previous section: we have 
\begin{equation}
\label{psiaw}
\psi = e^{i \theta} ( a + w)
\end{equation}
with $\langle w \rangle = 0$ and $a > 0$. The preservation of the $L^2$ norm guarantees that $a = \sqrt{1 - \Norm{w}{L^2}^2}$ as in \eqref{parseval}. We obtain exactly the same equations for $\theta$ - \eqref{thetamu} and $w$ - \eqref{eqw}, but with $\mu = 0$. 
The main point is to check that $F$ is Hamiltonian, and that we can apply a Birkhoff normal procedure, {\em i.e.} that the frequencies are generically non resonant.

We decompose the solution on the usual Fourier basis $\psi(t,x)= \sum_{j \in \Z^d} \psi_j(t) e^{ij \cdot x}$. The Hamiltonian $H(\psi,\psi^*)$ (see \eqref{hamil}) can be viewed as a function of the coefficients $\psi_j$ and $ \psi_{j}^*$, and we have 
$$
i \dot \psi_j = \frac{\partial H}{\partial \psi^*_j}( \psi, \psi^*). 
$$
Now \eqref{psiaw} is viewed as a change of variable for the Fourier modes: it defines a function $\psi \mapsto (a,\theta,w)$ with $w=(w_j)_{0 \neq j \in \Z^d}$, $0 \leq a \in \R$ and $\theta \in \R$ defined by
\[  \psi_0= a e^{i \theta} \ \ \ \text{and} \ \ \ \psi_j = w_j e^{i \theta} \ \ \ \text{for} \ \ \  j \in \mathcal{Z} := \Z^d \backslash \left\{ 0 \right\} . \]
In these new variables, the Hamiltonian function can be written 
$$
H(\psi,\psi^*) = \Hc(a,w,w^*), 
$$
as the Hamiltonian is gauge invariant ({\em i.e.} invariant by the transformation $\psi \mapsto e^{i\theta} \psi$), however this transformation is not symplectic. As $\psi_0 = e^{i \theta}a$, we obtain the following collection of equations: 
$$
\left|
\begin{array}{l}
\displaystyle i \dot{a} - \dot{\theta} a =  \frac{\partial \mathcal{H}}{\partial a} (a,w,w^*), \\[2ex] 
\displaystyle i \dot{w}_j - \dot{\theta} w_j = \frac{\partial \mathcal{H}}{\partial w_j^*} (a,w,w^*), \quad j \neq 0.
\end{array}
\right.
 $$
As $a\in \R_+$, taking the real part of the first equation in the previous system shows that
\begin{equation} \label{eq_theta}
  \dot{\theta}= - \frac{1}{2a} \frac{\partial \mathcal{H}}{\partial a} (a,w,w^*),   
  \end{equation}
so $\theta$ is well controlled by $a$ and $w$.
Inserting this equation in the equation for $w_j$, we then get the equation on the $j$-th mode:
\begin{equation} \label{j_mode_eq}
i \dot{w}_j =   \frac{\partial \mathcal{H}}{\partial w^*_j} (a,w,w^*) - \frac{w_j}{2a} \frac{\partial \mathcal{H}}{\partial a}(a,w,w^*).
\end{equation}
Now as $a = \sqrt{1 - \Norm{w}{L^2}^2}$ we also have
\[  \frac{\partial a}{\partial w^*_j} = \frac{- w_j}{2a},  \]
and therefore the equations of motion \eqref{j_mode_eq} for $w=(w_j)_{j \in \mathcal{Z}}$ are Hamiltonian, namely
\begin{equation} \label{equation_motion_logNLS}
i \dot{w}_j=\frac{ \partial \hat{H}}{\partial w^*_j}(w,w^*), \quad j \in \Zc, \quad \hat H(w,w^*) = \Hc(a, w , w^*), 
\end{equation}
with the real-valued Hamiltonian
\begin{align*}
 \hat{H}(w, w^*)&=  \Hc(a,w,w^*) \\
 &= \Norm{\nabla w}{L^2}^2 + \frac{\lambda}{(2\pi)^d}\int_{\T^d} |a + w|^2 (\log( a + w) + \log (a + w^*) - 1)  dx .	
 \end{align*}

As in \cite{faou2012}, we can use the expansion \eqref{lelog} for the logarithm as well as the expansion 
\[ a= 1+ \sum_{m\geq 1} \frac{\frac{1}{2} \left(\frac{1}{2}-1\right) \ldots \left( \frac{1}{2} - m +1\right)}{ m!} \left( \sum_{j \neq 0} |w_j|^{2} \right)^m\]
and
\[ \frac{1}{a}= 1+ \sum_{m\geq 1} \frac{-\frac{1}{2} \left(-\frac{1}{2}-1\right) \ldots \left( -\frac{1}{2} - m +1\right)}{ m!} \left( \sum_{j \neq 0} |w_j|^{2} \right)^m,   \]
to obtain a representation of the form 
\begin{align*}
\hat{H}(w, \overline{w}) &= \Norm{\nabla w}{L^2}^2 + \lambda  \Norm{\psi}{L^2}^2 + \frac{\lambda}{2}\frac{1}{(2\pi)^d} \int_{\T^d} (w^2 + \overline w^2) \dd w + \Pc(w,\bar w),
\end{align*}
\[ \text{ with } \ \Pc(w, \overline{w}) = \sum_{r\geq 3} \Pc_r(w,\overline{w}), \]
where $\Pc_r(w,\overline{w})$ denotes a polynomial of degree $r$ of the form
\[ \Pc_r(w,\overline{w}) = \sum_{p+q=r} \sum_{\substack{(j,\ell) \in \mathcal{Z}^p \times  \mathcal{Z}^q \\ \mathcal{M}(j,\ell)=0}} \Pc_{k,\ell} w_{j_1} \ldots w_{j_p} \overline{w}_{\ell_1} \ldots \overline{w}_{\ell_q},  \]
and
\begin{equation} \label{momentum}
 \mathcal{M}(j,\ell)= j_1 + \ldots + j_p -\ell_1 - \ldots - \ell_q
\end{equation}
denotes the momentum of the multi-index $(k,\ell)$, and that  the Taylor expansion of the Hamiltonian $\Pc$ contains only terms with zero momentum. Moreover, this Hamiltonian function is smooth on $\Bc_s(r_0) := \{ w \in H^s\, | \, \Norm{w}{H^s} \leq r_0\}$ for $r_0$ small enough owing to estimates of the form 
\[ |\Pc_{k,\ell}| \leq M L^{p+q},  \]
which is coming from the analyticity of the logarithm expansion \eqref{lelog} and of the function~$w \mapsto \sqrt{1 - \Norm{w}{L^2}^2}$. 
The equation for $w$ can thus be written (compare \eqref{eqw})
$$
i \partial_t w = \frac{\partial \hat H}{\partial w^*} (w, w^*) = - \Delta w + \lambda w + \lambda w^* + \frac{\partial \Pc}{\partial w^*} (w, w^*), 
$$
or equivalently, for $j \in \Zc$, 
$$
i \partial_t w_j = (|j|^2 + \lambda) w_j + \lambda \overline{w}_{-j} + \frac{\partial \Pc}{\partial \overline w_j} (w, \overline w), 
$$
owing to the fact that $(w^*)_{j} = \overline w_{-j}$. 

\subsection{Diagonalization and non-resonant frequencies}
The linear part in the differential equation \eqref{equation_motion_logNLS} for $w_j$ is $(|j|^2 + \lambda) w_j + \lambda \overline{w}_{-j}$. By using the same strategy as in the previous section, taking the equation for $w_j$ together with that for $\overline{w}_{-j}$, we are thus led to consider the matrix (for $n= |j|^2 \geq 1$)
\[ A_n= \left( \begin{array}{cc} n + \lambda & \lambda \\
                                - \lambda & -n-\lambda                        
                        \end{array} \right). \]
Hence we have the eigenvalues $\pm \sqrt{n^2+2 \lambda n}$ which are all real if we assume $\lambda  > -\frac{1}{2}$, which will be the case from now on.

\begin{lemma} \label{lemma_change_variables}
Let $\lambda >-1/2$. Then, for all $n \geq 1$, the matrix $A_n$ is diagonalized by a $2 \times 2$ matrix $S_n$ that is real symplectic and hermitian and has condition number smaller than 2:
\[ S_n^{-1} A_n S_n= \left( \begin{array}{cc} \Omega_n & 0 \\
                                0 & -\Omega_n                        
                        \end{array} \right)  
                        \ \ \ \text{with} \ \ \ \Omega_n= \sqrt{n^2+2 \lambda n}.  \]
\end{lemma}
\begin{proof}
By direct diagonalization we have
\[ S_n= \frac{1}{\sqrt{(n+\Omega_n)(n+2\lambda+\Omega_n) }} \left( \begin{array}{cc} n+\lambda + \Omega_n & -\lambda \\ - \lambda & n+\lambda + \Omega_n \end{array} \right)   \]
and
\[ S_n^{-1}= \frac{1}{\sqrt{(n+\Omega_n)(n+2\lambda+\Omega_n) }} \left( \begin{array}{cc} n+\lambda + \Omega_n & \lambda \\  \lambda & n+\lambda + \Omega_n \end{array} \right), \]
which are indeed real and self-adjoint. Denoting 
\[ J= \left( \begin{array}{cc} 0 & 1 \\  -1 & 0 \end{array} \right), \]
we can also easily check that $S_n$ and $S_n^{-1}$ are in fact symplectic, namely
\[ S_n^{T} J S_n = J \ \ \ \text{and} \ \ \ \left( S_n^{-1}\right)^{T} J S_n^{-1} =J.  \]
\end{proof}

Using the real symplecticity of the matrix $S_n^{-1}$ of Lemma \ref{lemma_change_variables}, we make the symplectic change of variables
\[ \left( \begin{array}{c} \xi_j \\ \overline{\xi}_{-j}  \end{array}  \right) = S_n^{-1}  \left( \begin{array}{c} w_j \\ \overline{w}_{-j}  \end{array}  \right)  \]
for $j \in \Z^d\backslash \left\{0 \right\}$ and $n=|j|^2 \geq 1$. This linear transformation applied to the Hamiltonian system \eqref{equation_motion_logNLS} gives the new Hamiltonian system
\begin{equation}
\label{new_eq_motion_logNLS} i \frac{\dd}{\dd t} \left( \xi_j(t)  \right) =  \frac{\partial \tilde{H}}{\partial \overline{\xi}_j} \left( \xi(t), \overline{\xi}(t)  \right) 
\end{equation}
and define a new Hamiltonian $\tilde{H}(\xi,\overline{\xi})=H(w,\overline{w})$ which wan be written under the form
\[ \tilde{H}(\xi,\overline{\xi})= H_0 + P =  \sum_{j \neq 0} \omega_j | \xi_j|^2 + P( \xi, \overline{\xi}),   \]
where we denote the frequencies $\omega_j=\Omega_n=\sqrt{n^2+2 \lambda n}$ for $|j|^2=n \geq 1$, and the non-quadratic term $P$ is of the form
\begin{equation} \label{eq_P}
  P( \xi, \overline{\xi}) = \sum_{p+q \geq 3} \sum_{\substack{(j,\ell) \in \mathcal{Z}^p \times  \mathcal{Z}^q \\ \mathcal{M}(j,\ell)=0}} P_{k,\ell} \xi_{k_1} \ldots \xi_{k_p} \overline{\xi}_{\ell_1} \ldots \overline{\xi}_{\ell_q},  
\end{equation}
where the sum is still only over multi-indices with zero momentum \eqref{momentum}, since the
transformation mixes only terms that give the same contribution to the momentum. From the properties of the original Hamiltonian $\Pc$ and the bound on the transformations $S_n$, we can state the following bound: 
for $k \in \mathcal{Z}^p$, $\ell \in \mathcal{Z}^q$,  the coefficients in \eqref{eq_P} are bounded by
\[ |P_{k,\ell}| \leq M L^{p+q},  \]
for some constants $M$ and $L$ independent of $p$ and $q$. 
%

\begin{remark}
\label{rem1}
The $H^s$-norm of the sequence $\xi$ is equivalent to the Sobolev norm of the perturbation of the $m$-th plane-wave of Theorem \ref{theorem_plane_wave} in the case $\rho = 1$ namely
\[ c \|\xi \|_{H^s(\T^d)} \leq \| \psi(t,.) - e^{i \theta} \|_{H^s(\T^d)} \leq  c \|\xi \|_{H^s(\T^d)},  \]
where $c$ and $C$ denote two positive constants depending on $\lambda$. In particular, under the assumptions of
Theorem \ref{theorem_plane_wave}, the system \eqref{new_eq_motion_logNLS} has small initial values whose $H^s$-norm is of order $\eps$.
\end{remark}

\subsection{Non resonance condition and normal form}

We are now going to prove that the frequencies $(\Omega_n)_n$ of our linearized system statisfy Bambusi's non-resonance inequality \cite{bambusi2003} for almost all values of $\lambda$:
\begin{lemma} \label{lemma_non_resonance_frenquencies}
Let $\lambda_- >-1/2$, $\lambda_+> \max(\lambda_-,1/2)$ and $r>1$. There exists $\alpha=\alpha(r)>0$ and a set of full Lebesgue measure $\Lambda \subset \left[ \lambda_-, \lambda_+\right]$ such that for every $\lambda \in \Lambda$ there is a $\gamma > 0$ such that, for all integers $p$, $q$ with $p+q \leq r$ and for all $m=(m_1, \ldots, m_p) \in \N^p$ and $n=(n_1, \ldots, n_q) \in \N^q$,
\begin{equation} \label{nonresonance}
    | \Omega_{m_1} + \hdots + \Omega_{m_p} - \Omega_{n_1} - \hdots - \Omega_{n_q} | \geq \frac{\gamma}{\mu_3(m,n)^{\alpha}},
\end{equation}
except if the frequencies cancel pairwise. Here, $\mu_3(m,n)$ denotes the third-largest among the integers $m_1, \ldots, m_p$, $n_1, \ldots, n_q$.
\end{lemma}
\begin{proof}
The proof is exactly the same as in the cubic nonlinear Schr\"odinger case, and can be found in \cite[Lemma 2.2]{faou2012}. Indeed, in this latter reference, the frequencies where under the form 
$\sqrt{n^4 + 2 n^2 \lambda_0 \rho^2 }$ where $\lambda_0$ was fixed and $\rho$ the varying parameter. The condition \eqref{nonresonance} is thus a consequence of this result by replacing $\lambda_0 \rho^2$ by the varying parameter $\lambda$ in the proof of \cite[Lemma 2.2]{faou2012}. 
Note that the arguments are similar to the one used in  \cite{bambusi2003,bambusi2006,xu1997}. 
\end{proof}
We are now in position to apply the normal form result, Theorem 7.2 of \cite{grebert2007} (see also \cite{bambusi2006}). 
We denote the ball of  radius $r > 0$: 
\[  \mathcal{O}_s(r) := \enstq{(\xi,\overline{\xi}) \in \C^{\mathcal{Z}} \times \C^{\mathcal{Z}}}{ \Norm{ \xi}{H^s }\leq r},  \]
where the Sobolev norm $\Norm{ \cdot}{H^s}$ can be written 
\[ \Norm{ \xi }{H^s} = \sum_{n \geq1} n^s J_n(\xi,\bar \xi) \ \ \ \text{with the super-actions} \ \ \ J_n(\xi,\bar \xi) =\sum_{|j|^2=n} |\xi_j|^2. \]
For a given Hamiltonian $K \in \mathcal{C}^{\infty}(\mathcal{O}_s(r),\C)$ satisfying $K(\xi,\overline{\xi}) \in \R$, we denote by $X_{K}(\xi,\overline{\xi})$ the Hamiltonian vector field
\[ X_k(\xi,\overline{\xi})_j= \left( i \frac{\partial K}{ \partial  \xi_j},-i \frac{ \partial K}{\partial \overline{\xi}_j}  \right), \ \ \ j \in \mathcal{Z},  \]
associated with the Poisson bracket
\[ \left\{ K, G \right\} = i \sum_{j \in \mathcal{Z}} \frac{\partial K }{\partial \xi_j} \frac{\partial G }{\partial \overline{\xi}_j} - \frac{\partial K }{\partial \overline{\xi}_j} \frac{\partial G }{\partial \xi_j}.   \]
We can now apply Theorem 7.2 of \cite{grebert2007}, which can be formulate into our settings with the following theorem:
\begin{theorem} \label{theorem_normal_form}
Let $\lambda$ be in the set $ \Lambda$ of full measure as given in Lemma \ref{lemma_non_resonance_frenquencies} for some $N \geq 3$. There exists $s_0$ such that for any $s \geq s_0$ there exists two neighborhoods $\mathcal{U} = \mathcal{O}_s(r_0)$ and $\mathcal{V} = \mathcal{O}_s(r_1)$ of the origin and a canonical transformation $\tau : \mathcal{V} \rightarrow \mathcal{U}$ which puts $\tilde{H}=H_0 +P$ in normal form up to order $N$, $i.e.$,
\[ \tilde{H} \circ \tau = H_0 + Z + R,   \]
where
\begin{itemize}
 \item $H_0=\sum_{j \in \mathcal{Z}} \omega_j | \xi_j|^2$,
 \item $Z$ is a polynomial of degree $N$ which commutates with all the super-actions, namely $\left\{ Z, J_n \right\}=0$ for all $n \geq 1$,
\item $R \in \mathcal{C}^{\infty}( \mathcal{V},\R)$ and $ \Norm{ X_R(\xi,\overline{\xi}}{H^s} \leq C_s \Norm{ \xi }{H^s}^N$ for $\xi \in \mathcal{V}$,
\item $\tau$ is close to the identity: $\Norm{  \tau(\xi,\overline{\xi}) - (\xi,\overline{\xi}) }{H^s} \leq C_s \Norm{ \xi }{H^s}^2$ for all $\xi \in \mathcal{V}$.
\end{itemize}
\end{theorem}
The verification of the hypotheses for our Hamiltonian $\tilde{H}$ is pretty standard and can be performed as in \cite{faou2012}. The consequence of this Theorem are well known from \cite{grebert2007,bambusi2006}: there exists $\varepsilon_0$ such that if $\Norm{\xi(0)}{H^s} \leq \varepsilon < \varepsilon_0$, then we have $\Norm{\xi(t)}{H^s} \leq 2 \varepsilon$ for a time $t \leq \varepsilon^{-N}$. As a consequence, the same result holds true for $w$ solution of the system \eqref{equation_motion_logNLS} up to changes of the constants. This implies that  $a = 1 + \mathcal{O}(\Norm{w}{L^2}^2) \leq C \varepsilon^2$ and $|\dot \theta| = \mathcal{O} \left(\Norm{w}{L^2}^2 \right) \leq C \varepsilon^2$ for $t \leq \varepsilon^{-N}$, and we can then conclude the proof owing to Remark \ref{rem1}. 
\section{Numerical simulations}
In this section we will perform a semi-discretization in time with a Lie-Trotter splitting method of the nonlinear Schr\"odinger-Langevin equation \eqref{SL_eq}. The operator splitting methods for the time integration of \eqref{SL_eq} are based on the following splitting 
 \[  \partial_t \psi = A( \psi ) + B ( \psi ) + C(\psi ),   \]
 where
 \[ A( \psi) =  i \Delta \psi, \ \ \ B ( \psi ) = -i \lambda \psi \log( |\psi|^2 ) , \ \ \ C(\psi )= -\frac{1}{2}  \mu  \psi \log\left( \frac{\psi}{\psi^*} \right), \]
and the solutions of the subproblems
\[ i  \partial_t u(t,x)=-   \Delta u(t,x), \ \ \ u(0,x)=u_0(x), \ \ \ x \in \T^d, \ \ \ t>0,      \]
\[ i  \partial_t v(t,x)= \lambda v(t,x) \log( v(t,x)) , \ \ \ v(0,x)=v_0(x), \ \ \ x \in \T^d, \ \ \ t>0,      \]
\[ i  \partial_t w(t,x)= \frac{ \mu}{2i}   w(t,x) \log\left( \frac{w(t,x)}{w^*(t,x)} \right), \ \ \ w(0,x)=w_0(x), \ \ \ x \in \T^d, \ \ \ t>0.      \]
The associated operators are explicitly given, for $t \geq 0$, by
\[  u(t,.)=\Phi^t_A(u_0)=e^{it \Delta} u_0, \]
\[  v(t,.)=\Phi^t_B(v_0)= v_0 e^{-it \log( |v_0|^2)} , \]
\[  w(t,.)=\Phi^t_C(w_0)=a_0 e^{i \theta_0 e^{- \mu t}}, \ w_0=a_0 e^{i \theta_0}. \]
Note that in our simulations the initial functions will be some small perturbations of plane waves and should stay away from zero over large times as induced by Theorem \ref{theorem_plane_wave} and Theorem \ref{theorem_plane_wave_SL}, so we do not need to saturate the logarithm nonlinearity by an $\eps$-approximation $\psi \mapsto \psi^{\eps} \log \left( |\psi^{\eps}|^2+ \eps \right)$ as performed in \cite{carles_bao_DF} or \cite{carles_bao_splitting}.\\

 All the numerical simulations will be made in dimension $d=1$ on the torus $\T = \left[-\pi ,\pi \right]$. The computation of $\Phi_A^t$ is made by a Fast Fourier Transformation. Let $K$ be a positive even integer and denote $\Delta x=2\pi/K$ and the grid points $x_j=-\pi+k \Delta x $ for $0 \leq k \leq K-1$. Denote by $\psi^{K,j}$ the discretized solution vector over the grid $(x_k)_{0 \leq k \leq K-1}$ at time $t=t_j= j \Delta t$, which can be written with the discrete Fourier ansatz
 \[  \psi^{K,j}(n)= \sum_{k =0}^{K-1} \psi_k^{K,j} e^{i n \cdot x_k}, \ \ \ 0 \leq n \leq K-1 .  \]
Let $\mathcal{F}_K$ and $\mathcal{F}_K^{-1}$ denote the discrete Fourier transform and its inverse, respectively. With this notation, $\Phi^t_A(\psi^{K,j})$ can be obtained by 
\[ \Phi^t_A(\psi^{K,j}) = \mathcal{F}_K^{-1} \left( e^{-i \Delta t ( \sigma^K)^2} \mathcal{F}_K( \psi^{K,j} )    \right),  \]
where
\[  \sigma^K = \left[ 0,1, \ldots, \left( \frac{K}{2} -1 \right), -\frac{K}{2}, \ldots,-1\right],  \]
and the multiplication of two vectors is taken as point-wise. In the following we will both plot the dynamics of the absolute value of the solution $\left(|\psi^{K,j}|\right)_{j}$ and the evolution of the actions $\left(|\psi_k^{K,j}|\right)_{j}$ for $0 \leq k \leq K-1$ over the discrete time $\left\{ 0, \ldots, T_{\max} \right\}$ with $T_{\max}=J \Delta t$ and $J \in \N^*$, which corresponds to the number of discretization points of time.  \\

In the following we perform our simulations with $K=2^{7}$ space discretization points on the interval $\T =\left[-\pi,\pi\right]$, using a time step $\Delta t =10^{-2}$ on the interval $\left[0,T_{\max} \right]$ with $T_{\max}=100$.  We also take the the initial function
\[  \psi_0(x)=\frac{1}{1 + 0.2 \cos(x)}. \]

\subsection{Evolution of the solution}
Here we plot the absolute value of the solution $\left(|\psi^{K,j}|\right)_{0\leq j \leq J}$. We take the constant $\lambda=0.5$ and we perform two simulations with $\mu$ either equal to 0 or 2. In the case without dissipation ($\mu =0$), we clearly observe the stability of the solution (Figure \ref{fig:mu_0_evolution.png}), whereas in the case with dissipation ($\mu =2$) we see that the solution converges quickly to its limit (Figure \ref{fig:mu_2_evolution.png}).

\begin{figure}[h]
	\centering
		\includegraphics[width=0.85\textwidth,trim = 0cm 1cm 0cm 1cm, clip]{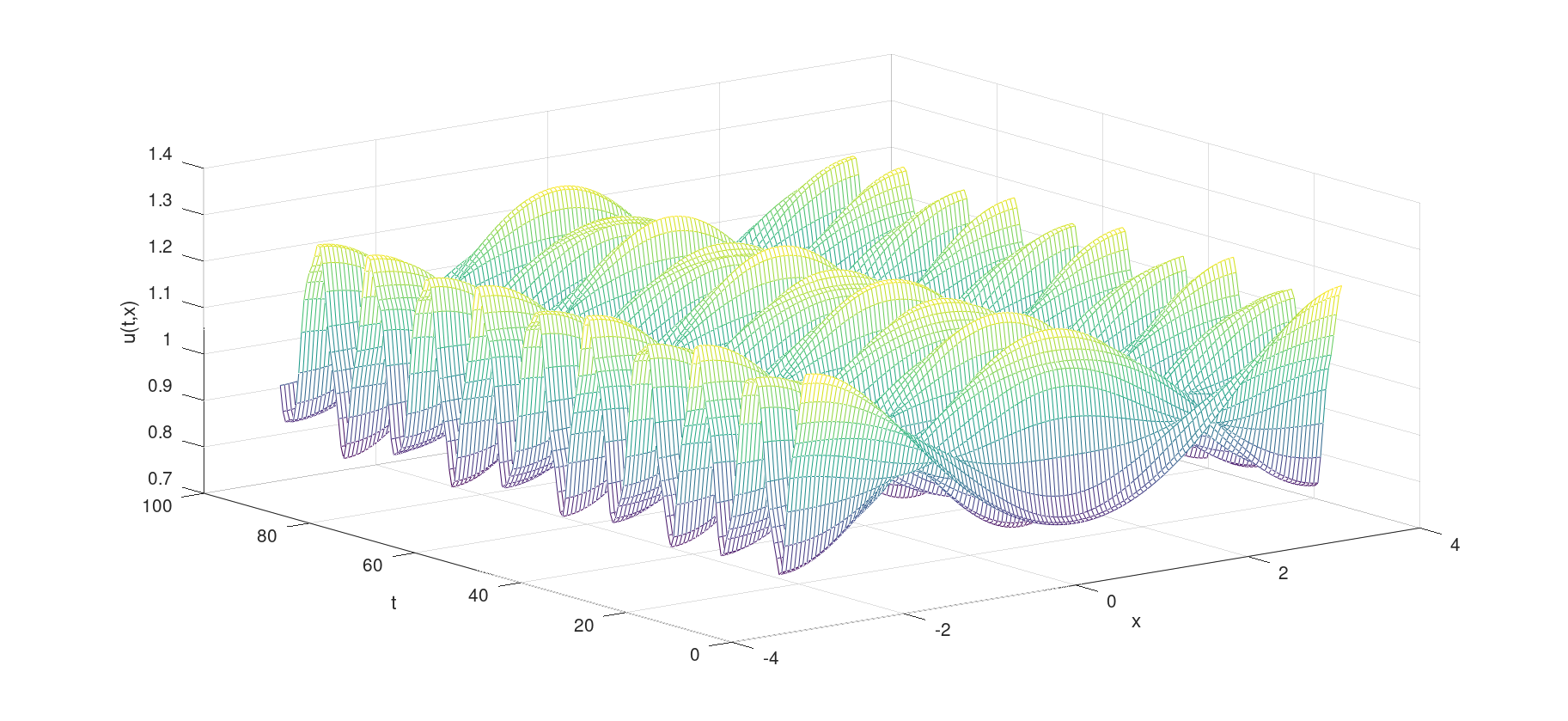}	
	\caption{Solution of equation \eqref{SL_eq} with initial datum $\psi_0$ ($\lambda=0.5$, $\mu=0$).}
	\label{fig:mu_0_evolution.png}
\end{figure}

\begin{figure}[h]
	\centering
		\includegraphics[width=0.85\textwidth,trim = 0cm 1cm 0cm 1cm, clip]{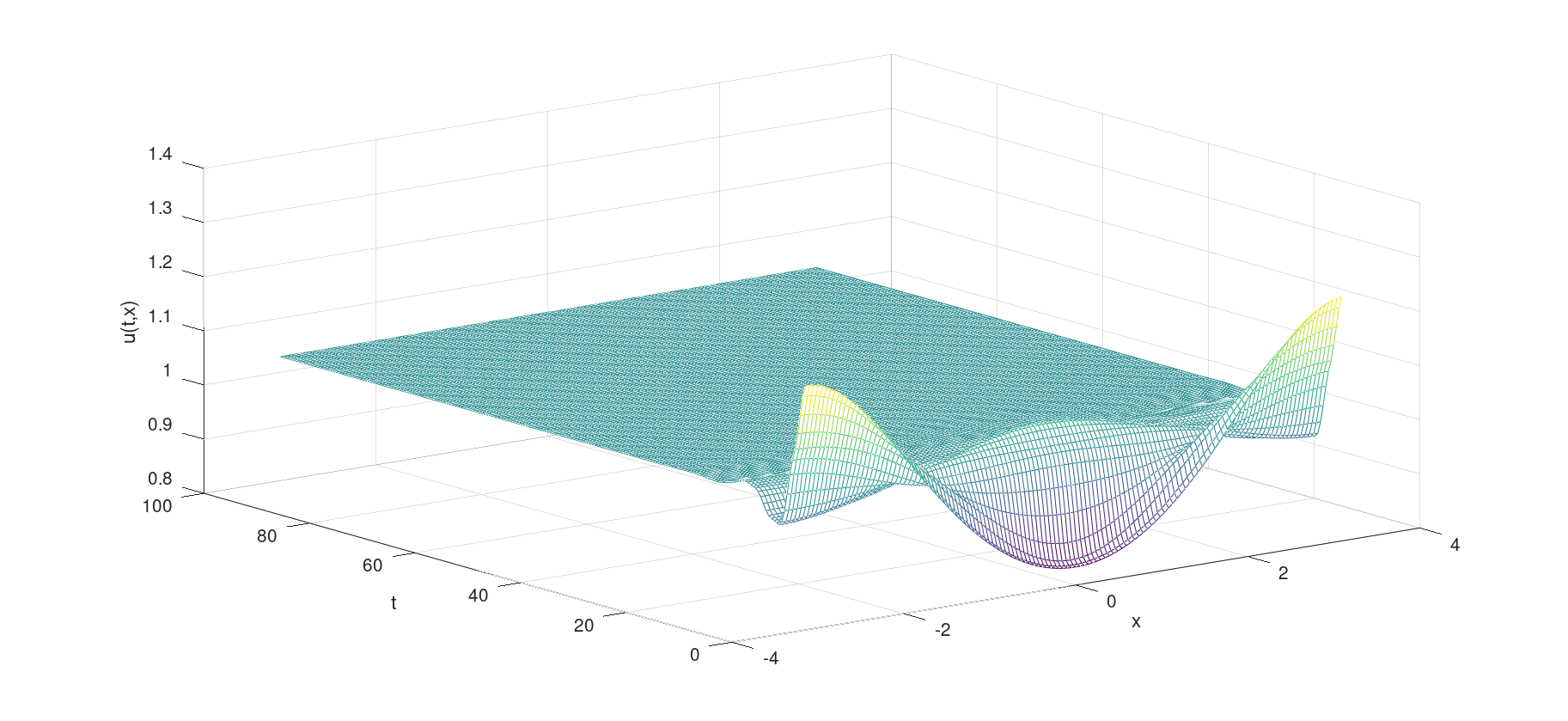}	
	\caption{Solution of equation \eqref{SL_eq} with initial datum $\psi_0$ ($\lambda=0.5$, $\mu=2$).}
	\label{fig:mu_2_evolution.png}
\end{figure}

\subsection{Evolution of the actions}
We now plot the evolution of the actions $\left(|\psi_k^{K,j}|\right)_{j}$ with $0 \leq k \leq K-1$, in logarithmic scale. We will first take the constant $\lambda=0.5$ and $\mu$ equal to $0$ (Figure \ref{fig:mu_0_actions.png}), $2$ (Figure \ref{fig:mu_2_actions.png}) and 8 (Figure \ref{fig:mu_8_actions.png}) in order to corroborate the state of Theorem \ref{theorem_plane_wave} and Theorem \ref{theorem_plane_wave_SL}. We then take a focusing constant $\lambda=-1$ with $\mu=2$ (Figure \ref{fig:lambda_minus_1_mu_2_actions.png}). \\

In Figure \ref{fig:mu_0_actions.png} we well observe the near-conservation of the actions of our stable solution. In Figure \ref{fig:mu_2_actions.png}, all the actions drop to zero at the same exponential rate, except the mode 0 which stays constant, as described in the proof of Theorem \ref{theorem_plane_wave_SL}. The actions still decays to zero in Figure \ref{fig:mu_8_actions.png}, but we observe that the first, second and third modes de not decrease at the same rate than the other ones, which corresponds to the cases $n=1$, $2$ and $3$ for the eigenvalues
\[ \alpha_n= \frac{\mu}{2} - \sqrt{\frac{\mu^2}{4} - n^2 - 2 \lambda n}, \]
and which confirms the analysis of Theorem \ref{theorem_plane_wave_SL}. Finally, the case of Figure \ref{fig:lambda_minus_1_mu_2_actions.png} which is not covered by our analys $(\lambda = -1)$ is way more exotic: the zero and first mode are conserved, and it also becomes true for the other modes after a transitory state where the actions decrease exponentially fast.

\begin{figure}[h]
	\centering
		\includegraphics[width=0.85\textwidth,trim = 0cm 1cm 0cm 1cm, clip]{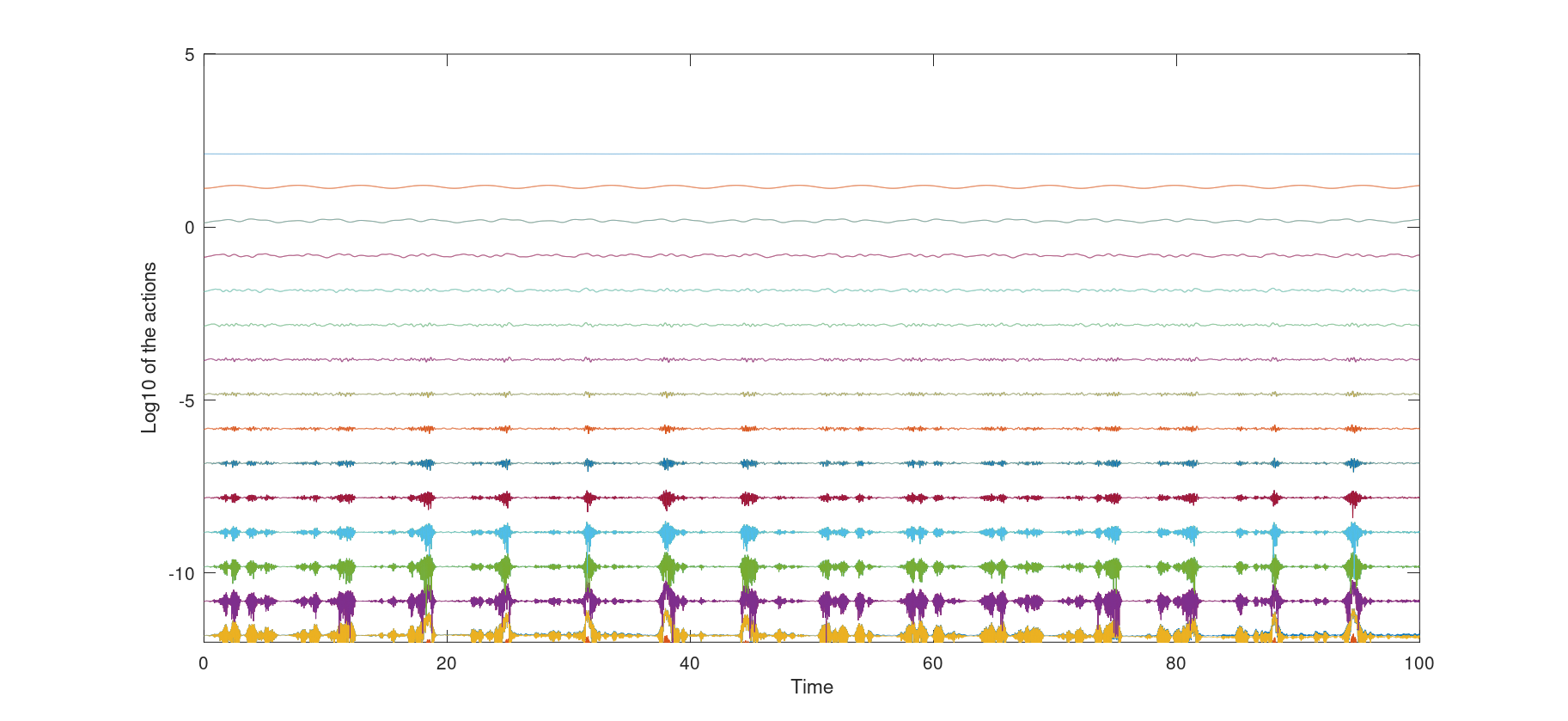}	
	\caption{Evolution of the actions of solution of equation \eqref{SL_eq} with initial datum $\psi_0$ ($\lambda=0.5$, $\mu=0$).}
	\label{fig:mu_0_actions.png}
\end{figure}

\begin{figure}[h]
	\centering
		\includegraphics[width=0.85\textwidth,trim = 0cm 1cm 0cm 1cm, clip]{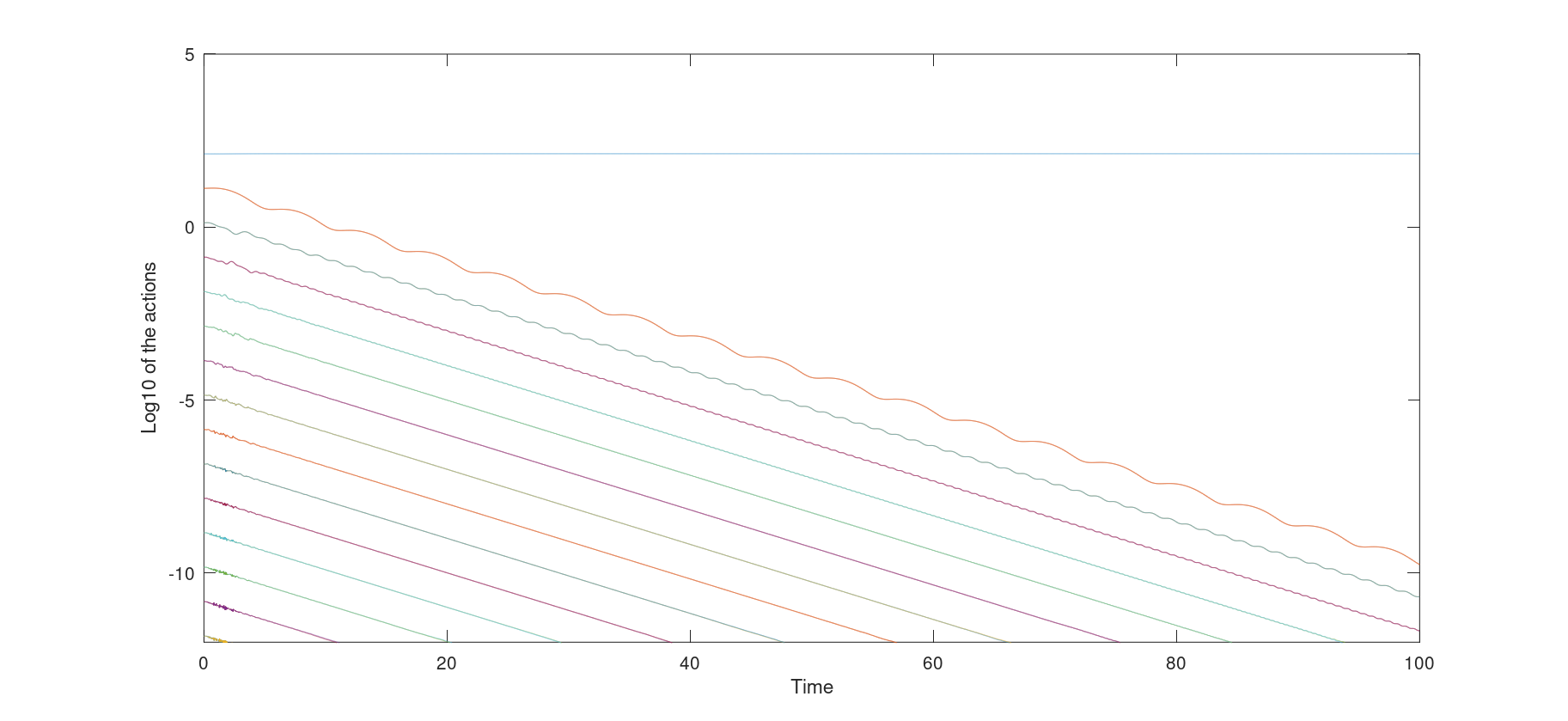}	
	\caption{Evolution of the actions of solution of equation \eqref{SL_eq} with initial datum $\psi_0$ ($\lambda=0.5$, $\mu=2$).}
	\label{fig:mu_2_actions.png}
\end{figure}

\begin{figure}[h]
	\centering
		\includegraphics[width=0.85\textwidth,trim = 0cm 1cm 0cm 1cm, clip]{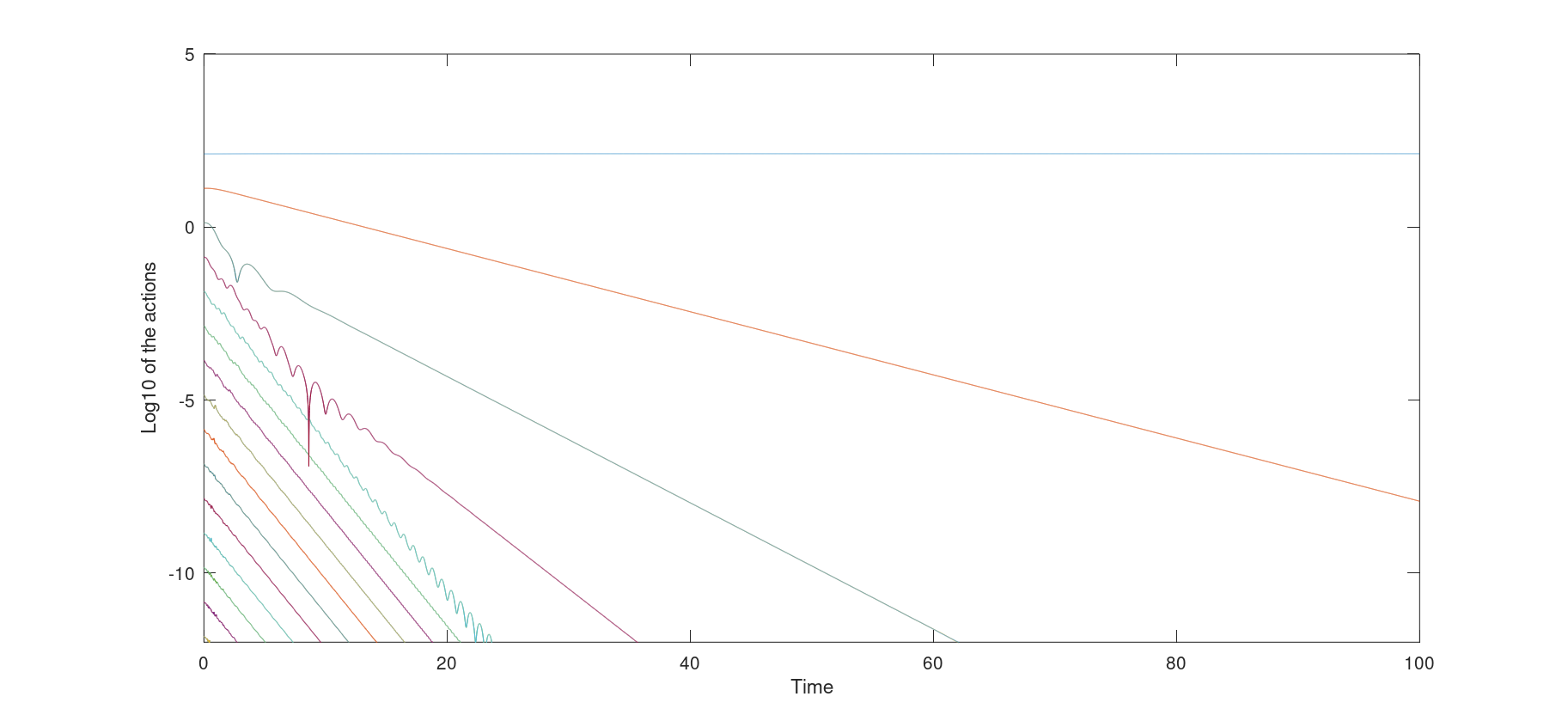}	
	\caption{Evolution of the actions of solution of equation \eqref{SL_eq} with initial datum $\psi_0$ ($\lambda=0.5$, $\mu=8$).}
	\label{fig:mu_8_actions.png}
\end{figure}

\begin{figure}[h]
	\centering
		\includegraphics[width=0.85\textwidth,trim = 0cm 1cm 0cm 1cm, clip]{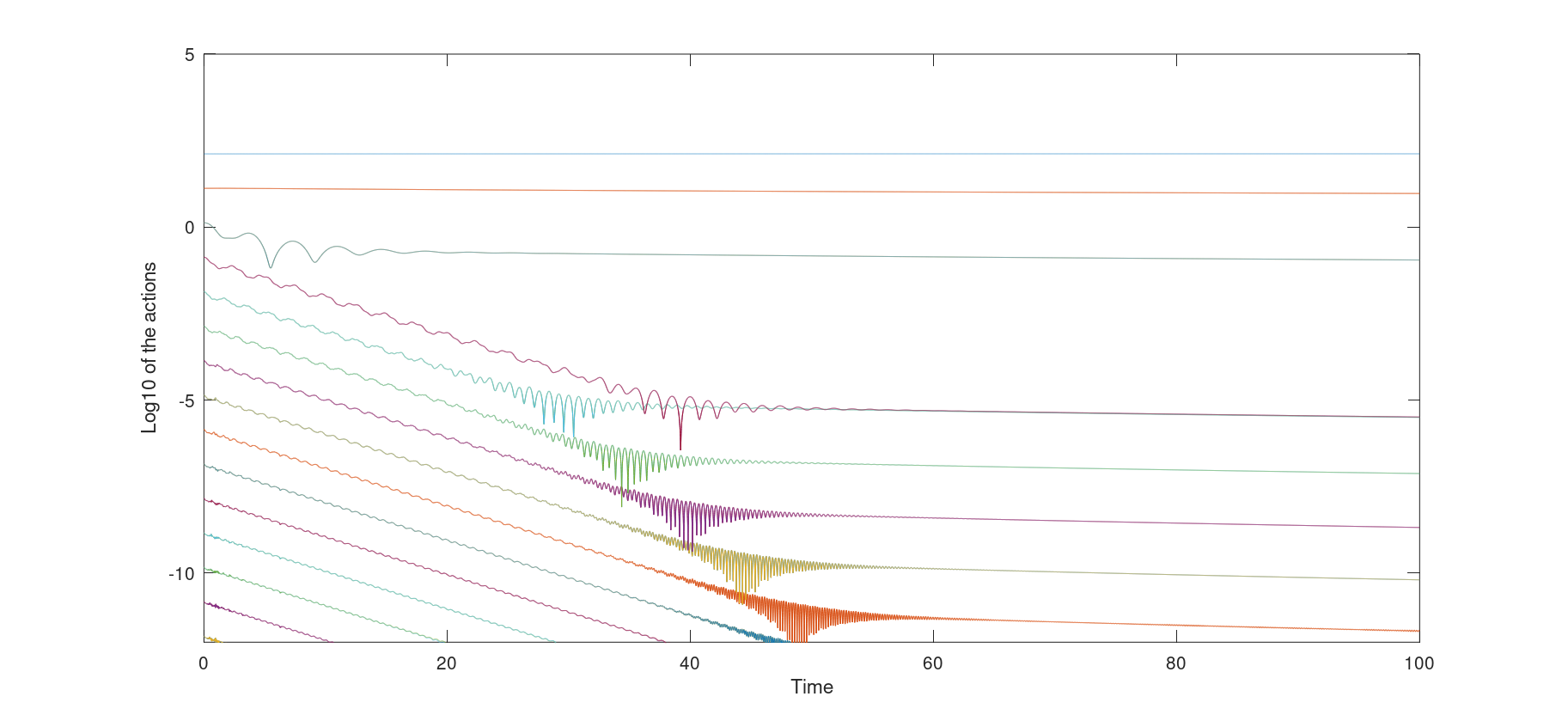}	
	\caption{Evolution of the actions of solution of equation \eqref{SL_eq} with initial datum $\psi_0$ ($\lambda=-1$, $\mu=2$).}
	\label{fig:lambda_minus_1_mu_2_actions.png}
\end{figure}

\bibliographystyle{plain}
\bibliography{biblio}

\end{document}